\documentclass[12pt]{amsart}

\title[Topologically 2-generated groups]{Topologically 2-generated groups}

\usepackage{amsfonts,amssymb,amsmath,amsthm}
\usepackage{url}
\usepackage{xcolor}
\usepackage{bbm}
\usepackage{anysize}

\usepackage[T1]{fontenc} 
\usepackage[utf8x]{inputenc}

\marginsize{1 in}{1 in}{1 in}{1 in}

\theoremstyle{plain}
\newtheorem{theorem}{Theorem}[section]
\newtheorem{lemma}[theorem]{Lemma}
\newtheorem{corollary}[theorem]{\bf Corollary}
\newtheorem{remark}[theorem]{\bf Remark}

\newtheorem{proposition}[theorem]{\bf Proposition}
\newtheorem{question}[theorem]{\bf Question}

\newcommand \Aut{ \mbox{Aut}}
\newcommand \Id{ \mbox{Id}}
\newcommand \SI{ S_\infty}
\newcommand \dom{{\rm{ dom}}}
\newcommand \rng{{\rm{ rng}}}
\newcommand \age{{\rm{ Age}}}
\newcommand \aut{{\rm{ Aut}}}

\newcommand \Iso{{\rm{ Iso}}}

\def\f{\mathcal{F}}
\def\rest{\restriction}

\newcommand \NN{\mathbb{N}}

\newcommand \QQ{\mathbb{Q}}
\newcommand \Ury{\mathbb{U}}

\newcommand \UU{\mathbb{U}_0}
\newcommand \ZZ{\mathbb{Z}}
\newcommand \SSS{\mathbb{S}}

\newcommand{\fra}{Fra\"\i ss\'e }

\author[A. Kwiatkowska]{Aleksandra Kwiatkowska}
\address{Institut f\"{u}r Mathematische Logik und Grundlagenforschung, Universit\"{a}t  M\"{u}nster,  
Einsteinstrasse 62,
48149  M\"{u}nster,
Germany {\bf{and}} 
Instytut Matematyczny, Uniwersytet Wroc{\l}awski,  pl. Grunwaldzki 2/4, 50-384 Wroc{\l}aw, Poland}
\email{kwiatkoa@uni-muenster.de}

\author[M. Malicki]{Maciej Malicki}
\address{Department of Mathematics and Mathematical Economics, Warsaw School of Economics, al. Niepodleg{\l}o\'sci 162, 02-554, Warsaw, Poland}
\email{mamalicki@gmail.com}

\thanks{The first named author was supported by Narodowe Centrum Nauki grant 2016/23/D/ST1/01097.}


\begin{document}

\begin{abstract}
We  prove that for a number of ultrahomogeneous structures $M$, including those with the free amalgamation 
property, the powers of the automorphism group $\Aut(M)^n$, $n=1,2,\ldots$,
and the group $L_0(\Aut(M))$ of measurable functions with values in $\Aut(M)$, have  a cyclically dense conjugacy class, in particular, are topologically 2-generated. This provides a number of new examples of groups with this property. Moreover, we will show that each of these groups is cyclically generated by a pair  generating the free group.
\end{abstract}

\maketitle

\section{Introduction}

For a Polish group $G$,  a tuple $(a_1,\ldots, a_n)$ in $G$  {\em topologically generates}
 $G$ if the closure of the subgroup of $G$ generated by it, is equal to $G$. In this situation, we say that $G$ is {\em topologically $n$-generated}.
There is a number of Polish groups which are topologically 2-generated in a very special way, specifically, which have a cyclically dense conjugacy class.
For a Polish group $G$ say that $(g,h)$ {\em (topologically) cyclically generates} $G$ if $\{g^l h g^{-l}\colon l\in~\ZZ\}$ is dense in $G$. 
In that case, we will say that $G$ has a  {\em cyclically dense  conjugacy class}.

Kechris and Rosendal (\cite{KR}, pages 314--315) gave  examples
of automorphism groups of countable structures which have a cyclically dense conjugacy class, among them was the automorphism group of the countable atomless
Boolean algebra. Macpherson \cite[Proposition 3.3]{Ma} proved that the automorphism group of the random graph has 
a cyclically dense conjugacy class,
and Solecki \cite[Corollary 4.1]{S} proved that the automorphism group of the rational Urysohn metric space 
has this property. It follows from the work of Glass, McCleary and Rubin \cite[Proposition 4.1]{GMR}
that the automorphism group of the random poset also has a cyclically dense conjugacy class.
Kaplan and Simon \cite{KS}, independently from our work, recently formulated   a model-theoretic condition 
on automorphism groups of countable structures to have a cyclically dense conjugacy class. 

A similar line of research is presented in Darji and Mitchell \cite{DM}, where the sets of topological 2-generators in the automorphism group of rationals, and in the automorphism group of the colored random graph, are studied.
It is shown there, in particular, that in each of these groups, for every non-identity element $f$, there is an element $g$ such that $(f,g)$ topologically 2-generates the whole group (see \cite[Theorems 1.3 and Corollary 1.8]{DM}.)

Groups of measurable functions (for the definition see the next section) were introduced by Hartman and Mycielski \cite{HM} and used to show
that every topological group can be embedded in a path connected one. In recent years, the structure 
and dynamics of those groups has been frequently studied, for example in \cite{FS, KM, PS, Sa, So}.
It seems that topological generators of groups of measurable functions have been  barely studied at all. Glasner~\cite{Gla} (see also Pestov \cite[Theorem 4.2.6]{P}) showed that $L_0(\SSS^1)$  is topologically 1-generated.

\section{Definitions and Results}

A Polish group is \emph{non-archimedean} if it admits a neighbourhood basis at the identity consisting of open subgroups.
It is well known that  non-archimedean Polish groups are exactly those that can be realized as automorphism groups of countable structures,
equipped with the pointwise convergence topology. Even more, non-archimedean Polish groups are exactly those that can be realized as automorphism groups of relational ultrahomogeneous countable structures (see \cite{BK}, pages 9-10),
where a structure $M$ is {\em ultrahomogeneous} if every automorphism between finite substructures of $M$ can be extended to an automorphism of the whole $M$. Therefore, without loss of generality, we will restrict our attention to automorphism groups of ultrahomogeneous structures.
If a countable and locally finite structure $M$ is ultrahomogeneous, then $\f=\age(M)$ -- the class  of all finite substructures embeddable in $M$ -- is a \fra class, i.e., it is a countable up to isomorphism class of finite structures which
 has the hereditary property HP (for every $A\in\f$, if $B$ is a substructure of  $A$, then $B\in\f$),
 the joint embedding property JEP (for any $A,B\in\f$ there is $C\in \f$ which embeds both $A$ and $B$),
 and the amalgamation property AP (for any $A,B_1,B_2\in\f$ and any embeddings $\phi_i\colon A\to B_i$, $i=1,2$,
 there are $C\in\f$ and embeddings $\psi_i\colon B_i\to C$, $i=1,2$, such that $\psi_1\circ \phi_1=\psi_2\circ \phi_2$).
Moreover, by the classical theorem due to Fra\"{i}ss\'{e}, for every \fra class  $\mathcal{F}$ 
of finite structures, there is a unique up to isomorphism countable ultrahomogeneous structure $M$ such that  $\mathcal{F}=\age(M)$. In that case, we call $M$ the {\em Fra\"{i}ss\'{e} limit} of $\mathcal{F}$, see \cite[Section 7.1]{Ho}. Because $\age(M)$ has AP, $M$ has  {\em the extension property}, that is, for any $X,Y\in\age(M)$, embeddings $i\colon X\to M$ and $\phi\colon X\to Y$, there is an embedding $j\colon Y\to M$ such that $ j\circ\phi=i$.

By a \emph{partial automorphism} of a structure $M$, we mean a mapping $p\colon A \rightarrow B$, where $A, B \subseteq M$ are finite, that can be extended to an automorphism of $M$. By $\aut(M)$, we denote the group of all automorphisms of $M$, equipped with product topology. For $X \subseteq M$, let $\aut_X(M)$ denote the subgroup of $\aut(M)$ that pointwise stabilizes $X$. We say that a structure $M$ has  {\em no algebraicity} if for any finite substructures $X, Y\subseteq M$, $X\cap Y=\emptyset$, the orbit $\aut_X(M). Y$ is infinite.
It is well known (see for example \cite[Theorem 7.1.8]{Ho}) that an  ultrahomogeneous structure $M$ has no algebraicity  if and only if  $\age(M)$ satisfies the \emph{strong amalgamation property}, i.e., we have additionally $\psi_1[B_1] \cap \psi_2[B_2]=\psi_1 \circ \phi_1[A]$ in the definition of the amalgamation property.
We say that $g\in\aut(M)$ has {\em no cycles} if for every $x\in M$ and $n\geq 1$, it holds that $g^n(x)\neq x$.

Let  $(X,\mu)$ be a standard non-atomic Lebesgue space (without loss of generality, $X$ is the interval [0,1] and $\mu$ is the Lebesgue measure), and let $Y$ be a Polish space. We  define 
$L_0(X,\mu; Y)$ to be the set of all ($\mu$-equivalence classes of) measurable (equivalently: Borel) functions from $X$ to $Y$. We equip this space with the convergence of measure topology. The neighborhood basis at $h\in L_0(X,\mu; Y)$ is given by sets of the form
\[[h,\delta,\epsilon]=\{ g\in L_0(X,\mu; Y) \colon \mu(\{x\in X \colon d(g(x),h(x))<\delta\})>1-\epsilon \}, \]
where $d$ is a fixed compatible complete metric on $Y$, and $\epsilon, \delta>0$.  We  equip $L_0(X,\mu; Y)$  with the metric
\[ \rho(f,g)=\inf\{\epsilon>0\colon \mu(\{ x\in X\colon d(g(x),f(x))>\epsilon\})<\epsilon \},\]

When $Y=G$ is a Polish group, then $ L_0(X,\mu; G)$ is a Polish group as well, with multiplication given by the pointwise multiplication in $G$: $(fg)(x)=f(x)g(x)$.  
For $g\in G$, we will denote by $f_g$ the constant function in  $L_0(G)$ with the value $g$. For notational reasons, we assume that the natural numbers $\NN$ start with $1$.

Our first result is the following theorem, which will provide a number of examples of topological groups that have a cyclically dense conjugacy class, in particular, that are topologically 2-generated.

\begin{theorem}\label{use}
Let $G$ be the automorphism group of an ultrahomogeneous structure $M$. Suppose that there exists $g\in G$ such that for any partial automorphisms $\phi_i\colon A_i\to B_i$ and $\psi_i\colon C_i\to D_i$, $i \leq n$, of $M$, there are $k, m \in\ZZ$ such that $g^{k}\phi_i g^{-k} \cup g^{m}\psi_i g^{-m}$ can be extended to a single automorphism of $M$, for each $i \leq n$.
Then each of the groups  $G^n$, $n \in \NN$, is cyclically generated by $((g,\ldots, g), \bar{b})$ for some $\bar{b} \in G^n$, and $L_0(G)$ is cyclically generated by $(f_g,b)$ for some $b\in L_0(G)$.

Moreover, if $M$ has no algebraicity and $g$ has no cycles, we can find such  $\bar{b} \in G^n$ 
($b\in L_0(G)$, respectively) such that $(g,\ldots, g)$ and $\bar{b}$ ($f_g$ and $b$, respectively)  generate the free group.
\end{theorem}

Let $\mathcal{F}$ be a class of finite structures in a relational language $L$. We say that $\mathcal{F}$  satisfies 
{\em the free  amalgamation property} if for every $A,B_1,B_2 \in \mathcal{F}$, and embeddings $\phi_i\colon A\to B_i$, $i \leq 2$, there is $C\in \mathcal{F}$ and there are embeddings $\psi_i\colon B_i\to C$, $i \leq 2$, such that we have:  

\begin{enumerate}
\item $\psi_1\circ\phi_1=\psi_2\circ\phi_2$,
\item $x\neq y$ for all $x\in S$ and $y\in T$,
\item for every relational symbol $R\in L$ of arity $n$, for every $n$-tuple $s=(s_1,\ldots, s_n)$ in~$C$,
if $R(s_1,\ldots,s_n)$, then  $s\subseteq \psi_1(B_1)$ or $s\subseteq \psi_2(B_2)$.
\end{enumerate}
Clearly, the free amalgamation property implies the strong amalgamation property.


The {\em random graph} (respectively, the {\em random triangle-free graph})  is the Fra\"{i}ss\'{e} limit of the class of all finite graphs (respectively, all finite triangle-free graphs). The {\em random tournament} is the  Fra\"{i}ss\'{e} limit of the class of all finite tournaments, where a tournament is a directed graph $D$ such that for any distinct $x,y\in D$ either there is an edge from $x$ to $y$ or there is an edge from $y$ to $x$, but not both.  
Furthermore, the {\em random poset} $\mathbb{P}$ is  the Fra\"{i}ss\'{e} limit of the class of all finite 
partially ordered sets, and the {\em rational Urysohn space} $\mathbb{U}_0$ is  the Fra\"{i}ss\'{e} limit of the class of all finite
metric spaces with rational valued distances.

\begin{theorem}\label{przyklady}
Suppose that: 
 \begin{itemize}
 \item[(i)] $M$ is a  countable relational ultrahomogeneous 
 structure  such that $\age(M)$ has the free amalgamation property or is $M$ is the random tournament, or
 \item[(ii)] $M$ is the rational Urysohn metric space, or
  \item[(iii)] $M$ is the random poset, or
  \item[(iv)] $M$ is  the rational numbers with ordering, or
    \item[(v)] $M$ is the countable atomless Boolean algebra, or the countable atomless Boolean algebra
    equipped with the dyadic measure.
 \end{itemize}

Then each $G^n$, $n \in \NN$, and $L_0(G)$, where $G={\rm Aut}(M)$, has a cyclically dense conjugacy class. In fact, each of these groups is cyclically generated by a pair  generating the free group.
\end{theorem}

A proof of Theorem \ref{przyklady} will be presented in Sections 3.2-3.5.
Each time we will verify that the assumptions of Theorem \ref{use} are satisfied.

It is easy to see that the class of graphs and the class of triangle-free graphs have the free amalgamation property, while the amalgamation in the class of finite tournaments is very canonical, but that class does not have the free amalgamation property. The classes of finite rational metric spaces, partial orderings, linear orderings, Boolean algebras, and Boolean algebras equipped with the dyadic measure, do not satisfy the free amalgamation property. Still, the conclusion of Theorem \ref{przyklady} holds for those structures.

For the case of the rational Urysohn space we use results of Solecki \cite{S}, and to deal with the random poset, we use the work of  Glass, McCleary and  Rubin \cite{GMR}. 

\section{Proofs}

\subsection{Proof of Theorem \ref{use}.}
For  $\bar{a}=(a^1, \ldots, a^{n})$,  an $n$-tuple in a Polish group $G$, let $f_{\bar{a}}$ be the
step function defined by
\[ f_{\bar{a}}(x)=a^k \mbox{ iff } x \in \left[ \frac{k-1}{n}, \frac{k}{n} \right]. \]
\begin{lemma}\label{gen}
Let $G$ be a Polish group and suppose that there is $g\in G$ such that for any $n\in\NN$ there is $\bar{h}\in G^n$ such that 
$((g,\ldots,g),\bar{h})$  cyclically generates  $G^n$. Then $L_0(G)$ is cyclically generated by $(f_g,b)$, for some $b\in L_0(G)$.
\end{lemma}

\begin{proof}

It is easy to see that for every Polish group $H$, if $(g,h)$ cyclically generates $H$, then the set
\[ A=\{h' \in H: (g,h') \mbox{ cyclically generates } H \} \]
is dense. As $A$ is  a $G_\delta$ set, it is in fact comeager. Thus, for a fixed $g \in G$ as in the statement of the lemma, and $n \in \NN$, every set
%
\[A_n=\{\bar{h}\in G^n\colon ((g,\ldots,g),\bar{h}) \text{ cyclically generates }  G^n\}\]
is comeager.

Regarding elements $\bar{h}$ of $A_n$ as step functions $f_{\bar{h}}$,  
 we identify $A_n$ with a subset of $L_0(G)$, which we will denote by $A'_n$.
We have that $A'_{2^1}\subseteq A'_{2^2}\subseteq \ldots $, and that $\bigcup_n A'_{2^n}$  is dense in $L_0(G)$.
This implies that for any non-empty open set $U$ in $L_0(G)$, the set
\[ \{h\in L_0(G)\colon \exists_{l\in\ZZ} \ f_g^l h f_g^{-l}\in U\} \]
is open and dense. Therefore the set
\[\{h\in L_0(G)\colon (f_g,h) \text{ cyclically generates } L_0(G) \}\]
is comeager, in particular non-empty.
\end{proof}

The lemma below is a  generalization to every dimension $n$ of a condition due to Kechris-Rosendal (\cite{KR}, page 314) on a non-archimedean Polish group to have a cyclically dense  conjugacy class.
 \begin{lemma}\label{gener}
Let $G$ be the automorphism group of an ultrahomogeneous structure $M,$ and 
let $n \in \NN$ and $\bar{g}=(g_1,\ldots, g_n)\in G^n$ be given. Suppose that for every choice of  partial automorphisms
$\phi_i\colon A_i\to B_i$ and  $\psi_i\colon C_i\to D_i$, $i \leq n$, of $M$
there are $k,m\in \ZZ$ such that for every  $i \leq n$, the partial automorphisms $g_i^k\psi_i g_i^{-k}$ and $g_i^m\phi_i g_i^{-m}$ have a common extension to an automorphism of $M$.
Then the set
 \[Y=  \{\bar{h} \in G^n\colon (\bar{g},\bar{h}) \text{ cyclically generates } G^n\} \]
 is comeager in $G^n$; in particular it is non-empty.
\end{lemma}

\begin{proof} 
Without loss of generality, we can assume that  $k=0$  in the statement of the lemma.

We fix $\phi=(\phi_1\colon A_1\to B_1,\ldots,  \phi_n\colon A_n \to B_n)$, and we will show that the set 
\[ Z_\phi=\{(h_1,\ldots, h_n)\in G^n\colon  \exists_m \forall_i \    g_i^m \phi_i g_i^{-m}\subseteq h_i\} \]  
is open and dense in $G^n$.

 Clearly, $Z_\phi$ is open. To show that it is also dense, fix $\psi=(\psi_1\colon C_1\to D_1,\ldots,  \psi_n\colon C_n\to D_n)$. The set
 $[\psi]\cap Z_\phi$ is non-empty, where 
  \[ [\psi]=\{(f_1,\ldots f_n)\in G^n \colon \forall_i \  \psi_i \subseteq f_i\}.\]
This follows from the assumption that there is $m$ such that for each $i \leq n$, the partial automorphisms $\psi_i\colon C_i\to D_i$  and $g_i^m\phi_i g_i^{-m}\colon g_i^m(A_i)\to g_i^m(B_i)$
have a common extension to an automorphism of $M$. 

Finally, notice that $Y\supseteq \bigcap_\phi Z_\phi$, hence $Y$ is comeager.
\end{proof}

We will use the following lemma  in Sections 3.2-3.4 to show that the automorphism group of each structure listed in Theorem \ref{przyklady} is cyclically generated by a pair that generates the free group.

\begin{lemma}\label{free}
Let $M$ be a countable structure with no algebraicity, and let $g\in\aut(M)$ be an automorphism with no cycles. Then for each
$n\in\NN$, the set
\[\{\bar{h}\in\aut(M)^n \colon (g,\ldots,g) \text{ and } \bar{h} \text{ generate the free group in } \aut(M)^n\}\]
and the set
\[\{h\in L_0(\aut(M))\colon f_g \text{ and } h \text{ generate the free group in } L_0(\aut(M))\}\]
are comeager.
\end{lemma}
\begin{proof}
For any non-trivial reduced word $w(s,t)$, since $M$ has no algebraicity and $g$ has no cycles, the set
\[\{h\in\aut(M)\colon \exists x\in M\ w(g,h)(x)\neq x\}\]
is open and dense. Hence,
\[B_w=\{h\in\aut(M)\colon  w(g,h)\neq {\rm Id}\}\] is open and dense. Put $B=\bigcap_w B_w$.
 
As the first set listed in the conclusion of the lemma contains $B^n$, it is comeager. Moreover, the second set contains $L_0(B)$, so, by  \cite[Lemma 5]{KM}, it is comeager as well.
\end{proof}

\begin{proof}[Proof of Theorem \ref{use}]
The theorem  follows immediately  from Lemmas \ref{gen},  \ref{gener} and \ref{free}.
\end{proof}

\subsection{The free amalgamation property and the random tournament}

In this section, we prove the following theorem.
\begin{theorem}\label{sapplus}
Let $M$ be a  countable relational ultrahomogeneous 
structure  such that $\age(M)$ has the free amalgamation property or $M$ is the random tournament.
Then each $G^n$, $n \in \NN$, and $L_0(G)$, where $G={\rm Aut}(M)$, has a cyclically dense conjugacy class.
In fact, each of these groups is cyclically generated by a pair  generating the free group.
\end{theorem}

We will need the following lemma.
\begin{lemma}\label{free-lemma}
Let $M$ be a  countable relational ultrahomogeneous 
structure in a language $L$
and suppose  that $\age(M)$ has the free amalgamation property or $M$ is the random tournament.
For each $p\colon X\to Y$, a partial automorphism of $M$ without cycles,
  and a finite set $P$,  there is a finite partial 
automorphism $q$ of $M$, extending $p$ and without cycles, and there is $k\in\mathbb{N}$
such that $P\cap q^k(P)=\emptyset$ and the following holds. 
For every relation $R\in L$ of arity $n$ and $s=(s_1,\ldots,s_n)\subseteq P\cup q^k(P)$ if 
$R(s_1,\ldots,s_n)$, then $s\subseteq P$ or $s\subseteq q^k(P)$, in the case of the free amalgamation property,
and for every $s\in P$ and $t\in q^k(P)$, $R(s,t)$, in the case of tournaments.
\end{lemma}

\begin{proof}

We recursively define the required $q$. We will do it in $k=2l-1$ steps, where $l$ is the maximum of the cardinalities
of orbits
of $p$.
Without loss of generality, extending $p$ and $P$ if necessary,  each orbit of $p$ has the cardinality $l$ and 
$X\cup Y=P$.
 Let $A_0=\dom(p)\setminus\rng(p)$ and $A_j=p^j(A_0)$, $j=1,\ldots,l$.
 
 We take $p_0=p$. For each $0<i\leq k$ we construct $p_i$ such that  each orbit of $p_i$ has the cardinality $l+i$
 and $\dom(p_{i+1})=\dom(p_i)\cup A_{l+i+1}$,
 where $A_j=p^j(A_0)$, $j=1,\ldots, l+i+1$.
Moreover, in the case of the free amalgamation,  for every relation $R\in L$ of arity~$n$ and $s=(s_1,\ldots,s_n)\subseteq  A_0\cup\ldots\cup A_{l+i}$ if 
$R(s_1,\ldots,s_n)$ and $s\cap A_0\neq\emptyset$, then $s\subseteq A_0\cup\ldots\cup  A_l=P$.
In the case of tournaments: for any $(s,t)\subseteq  A_0\cup\ldots\cup A_{l+i}$ such that $s\in A_0$ and 
$t\in A_l\cup\ldots\cup A_{l+i}$ it holds
$R(s,t)$.

Suppose that we have constructed $p_i$, $i<k$.
We now construct $p_{i+1}$. In the case of the free amalgamation, we amalgamate freely 
${\Id}\colon (A_1\cup\ldots\cup A_{l+i})\to (A_0\cup A_1\cup\ldots\cup A_{l+i})$ and 
$p_i^{-1}\colon (A_1\cup\ldots\cup A_{l+i})\to (A_0\cup A_1\cup\ldots\cup A_{l+i})$.
We can identify the free amalgam with 
${\Id}\colon A_0\cup\ldots\cup A_{l+i}\to A_0\cup\ldots\cup A_{l+i}\cup B$
and $p'_i\colon A_0\cup\ldots\cup A_{l+i}\to A_0\cup\ldots\cup A_{l+i}\cup B$, where $B\subseteq M$
is disjoint from $A_0\cup\ldots\cup A_{l+i}$ and  
$p'_i\colon A_0\cup\ldots\cup A_{l+i}\to A_1\cup\ldots\cup A_{l+i}\cup B$ is an isomorphism that extends $p_i$.
We let $p_{i+1}=p'_i$.
Then for every relation $R\in L$ of arity $n$ and $s=(s_1,\ldots,s_n)\subseteq A_0\cup\ldots\cup A_{l+i+1}$ if 
$R(s_1,\ldots,s_n)$ and $s\cap A_0\neq\emptyset$, then $s\subseteq A_0\cup\ldots  A_{l+i}$ and hence 
by induction
$s\subseteq A_0\cup\ldots  A_{l}$. We proceed similarly with tournaments. Instead of freely amalgamating 
${\Id}\colon (A_1\cup\ldots\cup A_{l+i})\to (A_0\cup A_1\cup\ldots\cup A_{l+i})$ and 
$p_i^{-1}\colon (A_1\cup\ldots\cup A_{l+i})\to (A_0\cup A_1\cup\ldots\cup A_{l+i})$, we require that the amalgam
${\Id}\colon A_0\cup\ldots\cup A_{l+i}\to A_0\cup\ldots\cup A_{l+i}\cup A_{l+i+1}$
and $p_i\colon A_0\cup\ldots\cup A_{l+i}\to A_0\cup\ldots\cup A_{l+i}\cup A_{l+i+1}$ satisfies:
for any $(s,t)\subseteq A_0\cup\ldots\cup A_{l+i+1}$ such that $s\in A_0$ and 
$t\in A_{l+i+1}$ it holds
$R(s,t)$, and hence by induction for any $(s,t)\subseteq A_0\cup\ldots\cup A_{l+i+1}$ such that $s\in A_0$ and 
$t\in A_l\cup\ldots\cup A_{l+i+1}$ it holds
$R(s,t)$.

Then $p_k$ is as required.

 


\end{proof}

Proposition \ref{bbaire} and Lemma \ref{free-lemma} will finish the proof of Theorem \ref{sapplus}.
\begin{proposition}\label{bbaire}
Let $n \in \NN$ be given. Then 
there is  $g\in G$ that has no cycles with the following property:
for every choice of  partial automorphisms
$\phi_i\colon A_i\to B_i$ and  $\psi_i\colon C_i\to D_i$, $i \leq n$, of $M$,
there is $k \in \ZZ$ such that for every  $i \leq n$, the partial automorphisms $g^k\psi_i g^{-k}$ and $\phi_i $ have a common extension to an automorphism of $M$.
\end{proposition}
\begin{proof}
Let $P=\bigcup_i(A_i\cup B_i)\cup\bigcup_i (C_i\cup D_i)$ be given and let $p$ be a partial automorphism of 
$M$, which does not have cycles. Apply Lemma \ref{free-lemma} to get $k$  and $q$  extending $p$ that
and are as in the conclusion of that lemma.
That implies that for each $i$, $\phi_i$, a partial automorphism of $P$, and 
$q^k\psi_i q^{-k}$, a partial automorphism of $q^k(P)$, have a common extension to an automorphism of $M$.

Therefore for given $\phi_i\colon A_i\to B_i$ and  $\psi_i\colon C_i\to D_i$, $i \leq n$, the set of  $h\in G$
such that there is $k \in \ZZ$ such that for every  $i \leq n$, the partial automorphisms $h^k\psi_i h^{-k}$, and $\phi_i $ have a common extension to an automorphism of $M$,
is dense in $T=\{h\in G\colon h \text{ does not have cycles}\}$. Moreover the set of such $h$'s is a  $G_\delta$
in $T$.
By the Baire category theorem, taking the intersection over all possible choices of 
 $\phi_i\colon A_i\to B_i$ and  $\psi_i\colon C_i\to D_i$, $i \leq n$, 
 we obtain that there exists the required $g$.
\end{proof}



\subsection{The rational Urysohn metric space}

Using Solecki \cite[Theorem 3.2]{S}, we can conclude from Theorem \ref{use}
that each $\Aut(\UU)^n$, $n \in \NN$, and  $L_0(\Aut(\UU))$, have a cyclically dense conjugacy class. 
 In this section, we show the following theorem.
\begin{theorem}\label{urysohn}
Each $\aut(\UU)^n$, $n\in\NN$, and $L_0(\aut(\UU))$,  has a cyclically dense conjugacy class. 
In fact, each of these groups is cyclically generated by a pair generating the free group.
 \end{theorem}
We point out that Solecki \cite[Corollary 4.1]{S} already
showed that each $\Aut(\UU)^n$, $n \in \NN$, has 
a cyclically dense conjugacy class.
 
 Let us first recall Solecki's theorem:

 \begin{theorem}[Solecki, Theorem 3.2 in \cite{S} ]\label{slawek}
 Let a finite metric space $A$ and a partial isometry $p$ of $A$ be given. Then there exist a finite metric space $B$ with $A\subseteq B$ as metric spaces,
 an isometry $\tilde{p}$ of $B$ extending $p$, and a natural number $m$ such that:
 \begin{itemize}
 \item[$(i)$] $\tilde{p}^{2m}=\Id_B$;
 \item[$(ii)$] if $a\in A$ has no cycles, then $\tilde{p}^j(a)\neq a$ for all $0<j<2m$;
 \item[$(iii)$] $A\cup \tilde{p}^m(A)$ is the amalgam of $A$ and $\tilde{p}^m(A)$ over $(Z(p), id_{Z(p)}, \tilde{p}^m\rest Z(p))$,
 where $Z(p)$ is the set of all cyclic points of $p$.
 \end{itemize}
 \end{theorem}
 
 \begin{remark}\label{acyc}
{\rm{  If  $p$ has no cycles, then $(iii)$ says that the distance between any point in $A$ and any point in $\tilde{p}^m(A)$  is equal to 
${\rm diam}(A)+{\rm diam}(\tilde{p}^m(A))=2{\rm diam}(A)$. Moreover, in that case, the partial isometry
\[ p_1=\tilde{p}\rest\left( \bigcup_{i=0}^{m-1} \tilde{p}^i(A) \right) \]
has no cycles as well.}}
\end{remark}
 
 \begin{proof}[Proof of Theorem \ref{urysohn}]
 We will verify that the assumptions of Theorem \ref{use} are satisfied.
We start with  enumerating $\UU$ into $\{ m_k\}_{k\geq 2}$ and 
 enumerating all finite tuples of pairs  of the same length of partial automorphisms of $\UU$ into $\{(p_{k,i}, q_{k,i})_{i}\}_{k\geq 2}$.
We then construct a  sequence $\{g_k\}_{k\geq 1}$, with $g_1$ equal to the empty function and $g_{k+1}$ extending $g_k$ for each $k$, of partial automorphisms
of $\UU$ such that for each $k$:
 \begin{enumerate}
 \item $g_k$ has no cycles;
 \item  $m_k$ is both in the domain and in the range of $g_{2k}$;
 \item  there is $m$ (that depends on $k$, but not on $i$) such that for every $i$, $g_{2k+1}^{m}p_{k,i} g_{2k+1}^{-m}\cup 
 q_{k,i} $ is a partial automorphism.
 \end{enumerate} 
 
 At even steps we make sure that (1) and (2) are satisfied, which is quite straightforward, and at odd steps that (1) and (3) are satisfied. Suppose that we constructed $g_{2k}$ and we want to construct $g_{2k+1}$ with the required properties.
 We have $p_{k,i}\colon A_{k,i}\to B_{k,i}$ and $q_{k,i}\colon C_{k,i}\to D_{k,i}$, $i \leq n$. 
 Let $A$ be the union of all the domains and ranges of $g_{2k}$, $p_{k,i}$, and $q_{k,i}$, $i \leq n$, and let $p=g_{2k}$.
 Apply Theorem \ref{slawek}, and get $B$, $\tilde{p}$ and $m$. We let 
 \[ g_{2k+1}=\tilde{p}\rest\left( \bigcup_{i=0}^{m-1} \tilde{p}^i(A) \right). \]
 Using observations  in Remark \ref{acyc}, we conclude that $g_{2k+1}$ has no cycles.
 We also see that for each $i \leq n$, 
 $g_{2k+1}^{m}p_{k,i} g_{2k+1}^{-m}\colon  g_{2k+1}^{m}(A_{k,i})\to g_{2k+1}^{m}(B_{k,i})$ and 
 $q_{k,i} \colon C_{k,i}\to D_{k,i}$, the map
 $g_{2k+1}^{m}p_{k,i} g_{2k+1}^{-m}\cup q_{k,i}$ is a partial automorphism. 
Indeed, as we observed  in Remark \ref{acyc}, the distance between a point in $\bigcup_i \left(C_{k,i}\cup D_{k,i}\right)$ and a point in 
 $\bigcup_i\left( g_{2k+1}^{m}(A_{k,i})\cup g_{2k+1}^{m}(B_{k,i})\right)$ is always the same.
 The required $g$ is $\bigcup_{k} g_{k}$.
 \end{proof}

\subsection{The random poset and other examples}
In this section, we verify Theorem  \ref{przyklady}(iii)-(v). We first focus on the random poset $\mathbb{P}$ and show the following.

\begin{theorem}\label{tposet}
Each $\aut(\mathbb{P})^n$,  $n\in\NN$, and $L_0(\aut(\mathbb{P}))$, has a cyclically dense conjugacy class.
In fact, each of these groups is cyclically generated by a pair  generating the free group.
\end{theorem}
It already follows from the work of  Glass-McCleary-Rubin \cite{GMR} that 
$\aut(\mathbb{P})$ has a cyclically dense conjugacy class whose elements generate the free group.

To prove Theorem \ref{tposet}, we will need Proposition \ref{lposet} and Lemma \ref{llposet}.
\begin{proposition}[Corollary 3.3 \cite{GMR}]\label{lposet}
There exists $g\in \aut(\mathbb{P})$ with the following two properties:
\begin{itemize}
\item[(1)] for every $x\in \mathbb{P}$, we have $x\leq g(x)$;
\item[(2)] there exists $a\in \mathbb{P}$ such that for any $x\in \mathbb{P}$ there is $n\in\ZZ$ with $g^n(a)\leq x\leq g^{n+1}(a)$. 
\end{itemize}
\end{proposition}

\begin{lemma}\label{llposet}
{\rm If $g\in \aut(\mathbb{P})$ is such as in Proposition \ref{lposet}, then we have:
For every $a\in \mathbb{P}$ and every $x\in \mathbb{P}$ there are $k,n\in\ZZ$ with $g^k(a)\leq x\leq g^{n}(a)$. }
\end{lemma}

\begin{proof}
Let $a_1$ and $x$ be given, and  let $a$ be as in (2) of Proposition \ref{lposet}. Then for some $m$, we have $g^m(a)\leq a_1\leq g^{m+1}(a)$. This implies
\[ \ldots g^{m-1}(a) \leq   g^{-1}(a_1)\leq g^m(a)\leq a_1\leq g^{m+1}(a)\leq g(a_1) \leq g^{m+2}(a)\leq \ldots \]
Therefore, if $n$ is such that $g^n(a)\leq x\leq g^{n+1}(a)$, since $g^{n-m-1}(a_1)\leq g^n(a)$ and $g^{n-1}(a)\leq g^{n-m-1}(a_1)$,
we get $g^{n-m-1}(a_1)\leq x \leq g^{n-m+1}(a_1)$.
\end{proof}

\begin{proof}[Proof of Theorem \ref{tposet}]
We verify that the assumptions of Theorem \ref{use} hold.
Let $g\in\aut(\mathbb{P})$ be as in Proposition \ref{lposet}, and fix  $n$ and 
   partial automorphisms of $P$, $\phi_i\colon A_i\to B_i$ and $\psi_i\colon C_i\to D_i$, $i \leq n$.
   Let $x\in \mathbb{P}$ be such that for any $z_1\in \bigcup_i (A_i\cup B_i)$, $z_1< x$; it exists by the extension property.
Then, using (1) and (2) in the properties of $g$ multiple times, get $m$ such that for every $z_2\in \bigcup_i (C_i\cup D_i)$,
$x\leq g^m(z_2)$. That implies that for every $z_1\in \bigcup_i (A_i\cup B_i)$ and $z_2\in \bigcup_i (C_i\cup D_i)$, we have
$z_1< g^m(z_2)$. That gives that for every $i \leq n$, $\phi_i $ and 
 $g^{-m}\psi_i g^{m}$ we can extend
 to a single automorphism of $M$.
\end{proof}

Let $\SI$ denote the group of all permutations of integers, 
$\aut(\QQ)$ the automorphism group of rationals,
$H(2^\ZZ)$ is the homeomorphism group of the Cantor set $2^\ZZ$, and $H(2^\ZZ,\mu^\ZZ)$ is the group of all $\mu^\ZZ$-preserving homeomorphisms of the Cantor set, where $\mu^\ZZ$ is the product of the measure $\mu$ on $2=\{0,1\}$ that assigns $\frac{1}{2}$ to each of 0 and 1.

\begin{theorem}\label{shift}
Let $G$ be one of the  $\SI$, $\aut(\QQ)$, $H(2^\ZZ)$ or $H(2^\ZZ,\mu^\ZZ)$. Then each  
$G^n$, $n\in\NN$, and $L_0(G)$, has a cyclically dense conjugacy class.
In fact, each of these groups is cyclically generated by a pair  generating the free group.
\end{theorem}
It already follows from the work of  Kechris-Rosendal \cite{KR} that each of the groups
  $\SI$, $\aut(\QQ)$, $H(2^\ZZ)$, and $H(2^\ZZ,\mu^\ZZ)$ has a cyclically dense conjugacy class.

\begin{proof}
Again, we verify that the assumptions of Theorem \ref{use} hold.
When $G=\SI$, we let $g$ in Theorem \ref{use} to be the shift map $n \mapsto n+1$. Similarly, when $G=\aut(\QQ)$, 
we take $g$ to be a shift by some rational number.

 If $G=H(2^\ZZ)$ (we identify $H(2^\ZZ)$ with the automorphism group of the Boolean algebra of all clopen sets in $2^\ZZ$)
 or $G=H(2^\ZZ,\mu^\ZZ)$ (we identify $H(2^\ZZ,\mu^\ZZ)$ with the group of measure preserving automorphisms of the Boolean algebra of all clopen sets in $2^\ZZ$), the $g$ equal to the Bernoulli shift will work.
Indeed, for any 
partial automorphisms of finite Boolean algebras $\phi_i\colon A_i\to B_i$ and $\psi_i\colon C_i\to D_i$, $i \leq n$,
we can find  $m\in\ZZ$ 
such that for each $i \leq n$, every atom in $A_i$ intersects every atom in $g^m(C_i)$,
and every atom in $B_i$ intersects every atom in $g^m(D_i)$. Then for every $i \leq n$, $\phi_i $ and 
 $g^{m}\psi_i g^{-m}$ we can extend
 to a single automorphism.
\end{proof}

\subsection{Final remarks}

It is not hard to see the following.
\begin{proposition}\label{down}
If $G$ is a Polish group and $L_0(G)$ is topologically $m$-generated, then $G^n$ is topologically $m$-generated for every $n \in \NN$.
\end{proposition}

\begin{proof}
Let $d$ denote a metric on $G$ and $\rho$ be the corresponding metric on $L_0(G)$.
Fix $n$, and suppose that $g_1,\ldots,g_m$ topologically generate $L_0(G)$. Let $(e_k)_{k\in\mathbb{N}}$ be a countable dense set in $G$. 
For each $s\in\mathbb{N}^n$ and $i\in\mathbb{N}$, we pick $\epsilon^s_{i}>0$ such that for each $s$, we have
$\sum_{s,i} \epsilon^s_{i}<\frac{1}{n}$. For each $s\in\mathbb{N}^n$ and $i\in\mathbb{N}$ now pick
 a word $l=l(s,i)\in F_m$,
such that  $\rho(l(g_1,\ldots, g_m) , f_s)<\epsilon^s_{i}$, where $f_s\in L_0(G)$ is such that $f_s(x)=e_{s(k)}$ if and only if 
$x\in (\frac{k}{n}, \frac{k+1}{n})$, $k=0,1,\ldots, n-1$.
Then there is a set $A_{s,i}$ of measure $\geq 1-\epsilon^s_{i}$, 
such that for $x\in A_{s,i}$
we have $d(l(g_1(x),\ldots, g_m(x)), f_s(x))<\epsilon^s_{i}$. 
This implies that for every $k$ the set $(\frac{k}{n}, \frac{k+1}{n})\cap\left(\bigcap_{s\in \mathbb{N}^n, i\in\NN} A_{s,i}\right)$ is non-empty, and 
for each $(x_0,\ldots, x_{n-1})$, where $x_i\in  (\frac{k}{n}, \frac{k+1}{n})\cap \left(\bigcap_{s\in \mathbb{N}^n,  i\in\NN} A_{s,i}\right)$, the tuple
$((g_1(x_0),\ldots, g_1(x_n)), \ldots, (g_m(x_0),\ldots, g_m(x_n)))$ topologically generates $G^n$,
in particular, that $G^n$ is topologically $m$-generated.
\end{proof}

We do not know whether the converse to Proposition \ref{down} holds. 
  \begin{question}
Let be $G$  a Polish group. Suppose that there exists $m$ such that for every $n$, $G^n$ is
 topologically $m$-generated. 
Is it the case that $L_0(G)$ is topologically $m$-generated?
\end{question}

So far we only dealt with automorphism groups of countable structures. Below we use our earlier 
results to conclude that certain ``large'' Polish groups 
have a cyclically dense
conjugacy class.

For the Lebesgue measure $\lambda$ on $[0,1]$ let
 $\aut([0,1],\lambda)$ denote the Polish group of all measure preserving transformations of the interval $[0,1]$.
 Let $\Iso(\Ury)$ be the Polish group of all isometries of the Urysohn metric space, where the Urysohn metric space is the unique Polish metric space 
which is ultrahomogeneous, and embeds isometrically every finite metric space.

Knowing that  the   groups $L_0(H(2^\ZZ,\mu^\ZZ))$ 
and $L_0(\aut(\UU))$ 
have a   cyclically dense conjugacy class,
we can deduce that the  groups $L_0(\aut([0,1],\lambda))$  and $L_0(\Iso(\Ury))$  
have the same property.

\begin{proposition}\label{metric}
If there is a continuous 1-to-1 homomorphism with a dense image of a Polish group $G$ into a Polish group $H$,
then there is a continuous 1-to-1 homomorphism with a dense image of a Polish group $L_0(G)$ into a Polish group $L_0(H)$.
In that case, if $L_0(G)$ has a cyclically dense conjugacy class, the same is true for $L_0(H)$.
\end{proposition}

\begin{proof}
 For a homomorphism  $f\colon G\to H$ of groups, the function 
$F\colon L_0(G)\to L_0(H)$ given by $F(h)(x)=f(h(x))$ is also a homomorphism of groups. 
Moreover, if $f$ is Borel, so is $F$ (see \cite{Mo}, the corollary on page 7), and hence it is continuous as long as $G,H$ are Polish (see Theorem 1.2.6 in \cite{BK}). Finally, if $f$ is 1-to-1, and has a dense image, the same holds for $F$.
Indeed, if $f$ has a dense image then for any $\epsilon>0$ and a Borel function $h\in L_0(H)$, since the set 
\[\{ (x,y)\in [0,1]\times H\colon d(h(x), y)<\epsilon \text{ and } y\in f[G] \}  \]
is analytic,  the density of the image of $F$ follows from the Jankov von Neumann uniformization theorem.
\end{proof}

\begin{corollary}
The groups $L_0(\aut([0,1],\lambda))$  and $L_0(\Iso(\Ury))$ 
have a cyclically dense conjugacy class.
\end{corollary}
\begin{proof}
It is well known that $H(2^\ZZ,\mu^\ZZ)$ is dense in $\aut([0,1],\lambda)$. Because the metric completion of the rational Uryshon space $\UU$ is the Urysohn space $\Ury$, and $\UU$ has the extension property, $\aut(\UU)$ is dense in $\Iso(\Ury)$ as well. Thus, we can apply Proposition \ref{metric} to $G=H(2^\ZZ,\mu^\ZZ)$  and  $H=\aut([0,1],\lambda)$, and then to $G=\aut(\UU)$ and $H=\Iso(\Ury)$) 
\end{proof}


\end{document}